\documentclass[10pt,twoside,a4paper]{article}

\usepackage{enumerate}
\usepackage[noadjust]{cite}
  
\usepackage{amssymb,amsmath,amsthm,color}
\usepackage{comment}
\newcommand\mytitle{$L_1$-Estimates for Eigenfunctions\\
     of the Dirichlet Laplacian}
\newcommand\lhead{M.~van den Berg, R.~Hempel,
and J.~Voigt}
\newcommand\rhead{ $L_1$-estimates for eigenfunctions}
\swapnumbers
\numberwithin{equation}{section}

\newtheorem{theorem}{Theorem}[section]
\newtheorem{corollary}[theorem]{Corollary}
\newtheorem{proposition}[theorem]{Proposition}
\newtheorem{lemma}[theorem]{Lemma}
\theoremstyle{definition}
\newtheorem*{definition}{Definition}
\newtheorem{remark}[theorem]{Remark}
\newtheorem*{remark*}{Remark}

\newtheorem{examples}[theorem]{Examples}

 \mathchardef\ordinarycolon\mathcode`\:
  \mathcode`\:=\string"8000
  \begingroup \catcode`\:=\active
    \gdef:{\mathrel{\mathop\ordinarycolon}}
  \endgroup

\def\N{{\mathbb N}}
\def\Z{{\mathbb Z}}
\def\R{{\mathbb R}}
\def\C{{\mathbb C}}

\def\rho{\varrho}
\def\3{\ss}
\def\theta{\vartheta}

\def\mid{\>;\>}
\def\mid{\,;\ }

\def\norm#1{\|#1\|}
\def\la{\lambda}
\def\eps{\varepsilon}
\def\epsilon{\varepsilon}

\def\Cci#1{C_c^\infty(#1)}
\def\phi{\varphi}

\def\sigmaess{\sigma_{\rm ess}}
\def\loc{{\rm loc}}

\def\supp{\hbox{\rm supp}\,}
\def\supp{\mathop{\rm spt}\nolimits}
\def\spt{\mathop{\rm spt}\nolimits}
\def\dist{\hbox{\rm dist}}
\def\diam{\hbox{\rm diam}}
\def\scapro#1#2{\left<#1,#2 \right>}
\def\iff{\quad \Longleftrightarrow \quad}

\def\ee{{\rm e}}

\def\di{{\,\rm d}}

\def\EE#1#2{E_#2(#1)}

\def\JJ{J}

\def\NN{{\mathcal N}}
\def\NN{N}
\def\OO{{\mathcal O}}

\newcommand\CC{{\mathcal C}}

\def\scapro#1#2{\left<#1\mskip2mu{,}\mskip2mu#2\right>}

\def\Hnought{H_0}
\newcommand\sse\subseteq
\renewcommand\emptyset\varnothing
\newcommand\indic{{\bf 1}}
\def\dparam{\Lambda}
\newcommand\class{set}
\newcommand\classes{sets}
\newcommand\trace{\mathop{\rm tr}\nolimits}
\newcommand\md{m_0}         
\newcommand\alphad{\alpha}  
\newcommand\cnought{c}      
\newcommand\vol{\mathop{\rm vol}\nolimits}
\newcommand\mC{{\hat C}}

\renewcommand\le{\leqslant}
\renewcommand\leq{\leqslant}
\renewcommand\ge{\geqslant}

\begin{document}
\title{\mytitle}

\author{Michiel van den Berg\footnote{School of Mathematics, University of Bristol,
University Walk, Bristol BS8 1TW, United Kingdom.},\
 Rainer Hempel\footnote{TU Braunschweig, Institute for Computational Mathematics,
 Am Fallersleber Tore 1, D-38100 Braunschweig, Germany.},\
 and J\"urgen Voigt\footnote{TU Dresden, Fachrichtung Mathematik,
D-01062 Dresden, Germany.}}

\date{}

\maketitle

 \begin{abstract}
 For $d \in \N$ and  $\Omega \ne \emptyset$ an open set in
 $\R^d$, we consider the eigenfunctions $\Phi$ of the Dirichlet
 Laplacian $-\Delta_\Omega$ of $\Omega$. If $\Phi$ is associated with an
 eigenvalue below the essential spectrum of $-\Delta_\Omega$ we
 provide estimates for the $L_1$-norm of $\Phi$
 in terms of its $L_2$-norm and spectral data.
 These $L_1$-estimates are then used in the comparison of the
 heat content of $\Omega$ at time $t>0$ and
 the heat trace at times $t' > 0$, where a two-sided estimate is established.
 We furthermore show that all eigenfunctions of $-\Delta_\Omega$ which are associated with a
 discrete eigenvalue of $H_\Omega$, belong to $L_1(\Omega)$.

\end{abstract}


\section*{Introduction}

We study the eigenfunctions $\Phi$ of the Dirichlet
Laplacian $H_\Omega$ on an open set $\Omega \sse \R^d$,
associated with a (discrete) eigenvalue $\la \in \R$. Our main interest is
to provide bounds on $\norm{\Phi}_1$, the norm of $\Phi$ in $L_1(\Omega)$, in terms of
the $L_2$-norm of $\Phi$ and spectral data. In many cases this is an improvement
over the elementary estimate $\norm{u}_1^2 \le \vol(\Omega) \norm{u}_2^2$,
valid for $\Omega$ of finite volume (and all functions $u \in L_2(\Omega)$, not
just eigenfunctions).
 Roughly speaking, we advocate here to replace the factor $\vol(\Omega)$ with
 $\la_1^{-d/2} N_{2\la_k}(H_\Omega)$, where $\la_1$ denotes the lowest eigenvalue of
 $H_\Omega$ and $N_t$ counts the (repeated) eigenvalues of
 $H_\Omega$ less than or equal to $t$. Our actual estimates are more
 complicated than that, and they only hold for eigenvalues below
 the essential spectrum. That spectral information can be used instead of
 volume doesn't come as a complete surprise: indeed, the uncertainty
 principle has found various expressions in spectral terms as in Weyl's Law
 and other well-known results that connect volumes in phase space
 with the counting of eigenvalues \cite{fef}. Estimates for the $L_1$-norm of
 eigenfunctions as presented here have been a desideratum for several
 decades now because they yield bounds on the heat content of $\Omega$
 in terms of the heat trace; see also \cite[p.\,2065]{Ob}.

Throughout this paper, $H_\Omega$ will be defined as the Friedrichs extension of
 $-\Delta$ on $C_c^\infty(\Omega)$;  $H_\Omega$ is self-adjoint and
non-negative.  More precisely, the form domain of $H_\Omega$ is given
by the Sobolev space $\Hnought^1(\Omega)$, and $H_\Omega$ satisfies
\begin{align}
    \scapro{H_\Omega u}{v} = \int_\Omega \nabla u \cdot \nabla v \di
    x,
\end{align}
for all $u \in D(H_\Omega)$ and all $v \in \Hnought^1(\Omega)$.
The eigenfunctions $\Phi$ of $H_\Omega$, associated  with an eigenvalue $\la$,
are
smooth and bounded and obey a well-known estimate
\begin{align}
    \norm{\Phi}_\infty^2 \le C \la^{d/2} \norm{\Phi}_2^2,
    \label{Linfty-bound}
\end{align}
with a constant $C$ depending on $d$ only. Theorem~\ref{2.12}
below gives an explicit constant. This estimate is a direct
consequence of the domain monotonicity of the heat kernel
\cite{dav}. Interpolation then yields bounds on $\norm{\Phi}_q$
for any $q \in (2,\infty)$. An immediate consequence of
\eqref{Linfty-bound} is a lower bound for $\norm{\Phi}_1$ of the
form
 \begin{align}
    \norm{{\Phi}}_1^2 \ge   C \la^{-d/2} \norm{{\Phi}}_2^2,
    \label{lowerbound}
\end{align}
where $C$ is a strictly positive constant depending on $d$ only.
Note that the estimates \eqref{Linfty-bound} and
\eqref{lowerbound} hold for \emph{all} eigenvalues.

We now complement the $L_\infty$-estimate \eqref{Linfty-bound} by upper bounds on
$\norm{\Phi}_1$, where we have to make the stronger assumption
that $\Phi$ is an eigenfunction associated with a discrete eigenvalue
$\la_k$ located below the essential spectrum of $H_\Omega$.
Here the $\la_k$ are numbered in increasing order and repeated according to their
respective multiplicity. One of our basic estimates reads as follows:

 \vskip1em

\begin{theorem} \label{Theorem 0.1}
 For any $d \in \N$ there exists a
 constant $C$ (depending on $d$ only) with
the following property: If $\Omega \ne \emptyset$ is an open subset of $\R^d$ with
 $\sigmaess(H_\Omega) = \emptyset$, we have
\begin{align}
    \norm{{\Phi}}_1^2 \le   C \lambda_1^{-d/2}
\left( \Bigl(\frac{\la_k}{\la_1}\Bigr)^d
     (\log\NN_{2\la_k}(H_\Omega))^d \NN_{2\la_k}(H_\Omega)
     + \Bigl(\frac{\la_k}{\la_1}\Bigr)^{4d -3} \right)  \norm{{\Phi}}_2^2,
   \label{upperbound}
\end{align}
 for all eigenfunctions ${\Phi}$ of $H_\Omega$ associated with the
eigenvalue $\la_k$.
 \end{theorem}

\vskip1em

   A slightly more general version is given in Corollary \ref{2.5}.
The estimate \eqref{upperbound} contains three factors which we
believe are
 essential: there is the factor $\la_1^{-d/2}$  which is due to scaling, as can
be seen from the ground state eigenfunctions of a ball of radius $r > 0$.
The presence of $\NN_{2\la_k}(H_\Omega)$ will become more clear later on.
Factors containing $\la_k/\la_1$ deal with the behavior of the
estimate for large eigenvalues as compared to small eigenvalues.

 The discrepancy between the lower bound  \eqref{lowerbound} and the
 upper bound \eqref{upperbound} is, at least in part, due to the fact
 that the estimate \eqref{lowerbound} seems
 to be far off in situations with large clouds of eigenvalues close to
 $\la_k$, as can be seen in simple examples like our Example \ref{2.14}(3).
 However, we do not expect the estimate \eqref{upperbound} to be optimal in 
any
 respect.

  There are similar, but somewhat more complicated estimates for the case
where $\sigmaess(H_\Omega) \ne \emptyset$ and
 $ \la_k < \inf\sigmaess(H_\Omega)$; cf.~Corollary \ref{2.6}.
Along these lines it would also be possible to give estimates for
$\norm{{\Phi}}_1$
 that involve the gap length $\la_{k+1} - \la_k$, provided $\la_{k+1} > \la_k$, for
eigenfunctions $\Phi$ of $H_\Omega$ associated with the eigenvalue
$\la_k$. We refer to the comment after Corollary
\ref{2.5}. We emphasize that our bounds do {\it not} depend on the
volume of $\Omega$.

We furthermore show that the eigenfunctions $\Phi$ of $H_\Omega$
belong to $L_1(\Omega)$ if they are associated with a discrete
eigenvalue $\la$, and so  $\Phi \in \bigcap_{1 \le p \le
\infty}L_p(\Omega)$. There is no reason to expect a similar result
for eigenfunctions associated with an eigenvalue which belongs to
the essential spectrum; see e.g.\ Example~\ref{2.14}(2) below.

 Let us briefly indicate why estimates as in Theorem~\ref{Theorem 0.1}
 are possible.  They are essentially based on three facts which we
 describe in terms of a covering of $\R^d$ by cubes $Q_{n,j}$ (for $n \in \N$ and
 $j \in \Z^d$), where each $Q_{n,j}$ has edge length $2n$ and is centered at $nj$.
 Let us consider $\Omega$, $\la_k$ and $\Phi$ as in Theorem~\ref{Theorem 0.1}.

 (1) We first observe that there can only be a finite number
 of cubes $Q_{n,j}$ such that the Dirichlet Laplacian of $\Omega \cap Q_{n,j}$
 has eigenvalues below  $2 \la_k$. In fact, the number of these cubes can
 be estimated in terms of $N_{2\la_k}$. We let $F_n$ denote the union
 of these cubes. The contribution to the $L_1$-norm of $\Phi$ coming from
 $F_n$ (or a slightly larger set) can now be estimated in terms of $\vol(F_n)$ and, thus, in terms
 of $n$ and $N_{2\la_k}$.

 (2) Letting $G_n := \Omega \setminus F_n$, we use a partition of unity
 (subordinate to the covering by the cubes $Q_{n,j}$) and the IMS-localization
 formula to show that the Dirichlet Laplacian of $G_n$ has no eigenvalues
  below $3\la_k/2$.

 (3) The third fact concerns the decay of the eigenfunctions
 $\Phi$ associated with the eigenvalue $\la_k$ as we move away from $F_n$.
 Here we use a rather precise, quantitative version of exponential decay which takes into
 account the distance from the set $F_n$. A standard decay estimate holding
 just outside some large ball containing $F_n$ would clearly be insufficient
 for our purposes. Since exponential decay takes place on a larger
 scale we also have to introduce the sets $\tilde F_n$ and $
{\tilde{\tilde{F}}}_n$ that are somewhat larger than $F_n$.

 The estimate given in Theorem~\ref{Theorem 0.1} can be applied in the
comparison of the \emph{heat content}
 $Q_\Omega$ of an open set $\Omega \subseteq \R^d$,
\[
    Q_\Omega(t) := \int_\Omega\int_\Omega p_\Omega(x,y;t)\di y\di x
\]
at time $t > 0$, where $p_\Omega\colon \Omega\times\Omega\times(0,\infty)\to
[0,\infty)$ is the Dirichlet heat kernel for $\Omega$,
and the \emph{heat trace} $Z_\Omega$,
\[
    Z_\Omega(t) :=  \sum_{k=1}^\infty \ee^{-\la_k t}
\]
at time $t > 0$, where it is assumed that $H_\Omega$ has compact resolvent and
that $(\la_k)$ is the sequence of all eigenvalues of $H_\Omega$.
We shall show in Section~\ref{content-trace} that
$Q_\Omega(t) < \infty$ for all
 $t > 0$ is equivalent to $Z_\Omega(t) < \infty$ for all $t > 0$, and that
 there is a two-sided estimate. Note that our
 upper bound for $Q_\Omega(t)$
involves $Z_\Omega(t/2)$ and $Z_\Omega(t/6)^3$. The interest in
a bound on $Q_\Omega$ in terms of $Z_\Omega$ lies in the fact that
the quantity $Z_\Omega$ is much simpler (and also simpler to
compute) than $Q_\Omega$ because only information from the Hilbert
space $L_2(\Omega)$ is needed.

 Similar estimates on the $L_1$-norm of eigenfunctions of Schr\"odinger
 operators will be the subject of a forthcoming paper.
  It would also be of great interest and importance
 to generalize our results to the case of domains on Riemannian manifolds
 with sub-exponential growth at infinity \cite{S}.
 The case of hyperbolic manifolds poses different
 challenges as can be seen from \cite{DST}.

Our paper is organized as follows. In Section 1 we first consider the
case where $\la_1(H_\Omega) = 1$ and present a basic estimate in its
most general form (viz.~Proposition 1.1); the proof of Proposition 1.1
is deferred to Section 2. We then derive several estimates
from Proposition 1.1 by scaling and a judicious choice of
parameters in Proposition 1.1. In Section 2 we construct IMS-partitions of
unity, depending on a parameter $n$, and prove Proposition 1.1;
here we rely on an exponential decay estimate stated in Lemma 2.3.
Section 3 is devoted to a proof of Lemma 2.3.

In Section 4 we combine some results from the theory of the
Laplacian in $L_p(\Omega)$, defined as the generator of the heat
semigroup acting in $L_p(\Omega)$, to show that the Riesz
projection, associated with the eigenspace of a discrete
eigenvalue, is independent of $p$, for $1 \le p < \infty$. It is
then easy to conclude that the range of this projection must be
contained in $L_1(\Omega)$. This part of the paper has been
motivated by the work \cite{hem-voi, he-vo-87} of two of the
authors on the $L_p$-spectrum of Schr\"odinger operators.

In Section 5, finally, we discuss a two-sided estimate for the
heat content and the heat trace, using the kernel $ p_\Omega(x,y
;t)$ of the heat semigroup. We also give a proof of the lower
bound of Theorem \ref{2.12}.
\bigskip

{\bf Disclaimer.} In much of this text we let $C$ denote a generic
non-negative constant the value of which may change from line to line.
\bigskip

{\bf Acknowledgement.} The authors are indebted to Hendrik Vogt for useful 
discussions.

\vskip1em
\section{Estimates for the $L_1$-norm of eigenfunctions}
\label{estimates}

Let $d \in \N$. For an open set $\emptyset \ne \Omega \sse \R^d$ we
let $H_\Omega$ denote the (self-adjoint and non-negative) Dirichlet Laplacian
of $\Omega$, i.e., $H_\Omega$ is the unique self-adjoint operator with
form domain given by the Sobolev space $\Hnought^1(\Omega)$, and satisfying
\begin{align*}
   \scapro{H_\Omega u}{v} =
    \scapro{\nabla u}{\nabla v}
  \qquad (u \in D(H_\Omega),\ v \in \Hnought^1(\Omega)),
\end{align*}
where $D(H_\Omega) \sse \Hnought^1(\Omega)$ denotes the domain of
$H_\Omega$.
By construction,  $C_c^\infty(\Omega)$ is a form core of $H_\Omega$.
Furthermore, $D(H_\Omega) \sse H^2_{\loc}(\Omega)$
 and $H_\Omega u = -\Delta u \in L_2(\Omega)$ for any $u \in D(H_\Omega)$.
In general, $D(H_\Omega)$ need not be contained in the Sobolev space
$H^2(\Omega)$.

\vskip1em

\begin{definition} For $\dparam > 0$, we let ${\cal O}_\dparam$ denote the
{\class}
of
all
open sets $\Omega \sse \R^d$ that enjoy the property
\begin{align*}
  \dparam = \inf\sigma(H_\Omega) < \inf\sigmaess(H_\Omega).
\end{align*}
\vskip1em
The sets $\Omega \in {\cal O}_\dparam$ may be unbounded, and they may have
infinite volume. Also, they may consists of
countably many components.
No regularity of the boundary $\partial\Omega$ will be required.
\end{definition}

The spectrum of $H_\Omega$ below $\Sigma_\Omega := \inf\sigmaess(H_\Omega)$ is
purely discrete.
It consists of a countable set of eigenvalues $\lambda_k(H_\Omega)$
(with $1 \le k \le K$ for some $K \in \N$, or for $k \in \N$), which we assume
to be numbered such that
\begin{align*}
   \dparam = \la_1(H_\Omega) \le \la_2(H_\Omega) \le \ldots \le
   \la_k(H_\Omega) \le  \la_{k+1}(H_\Omega) <  \Sigma_\Omega,
\end{align*}
 where each eigenvalue is repeated according to its multiplicity.
 If there is an infinite number of eigenvalues, we
 have $\la_k \to
 \Sigma_\Omega$ if $\sigmaess(H_\Omega) \ne \emptyset$, and
 $\la_k \to \infty$ if $H_\Omega$ has compact resolvent.
 If $\Omega$ is connected the ground state eigenfunction is unique (up to scalar multiples)
 and $\la_1 < \la_2$.

\smallskip
 Our results pertain in particular to the case where $H_\Omega$ has compact
 resolvent. If $\Omega$ has finite volume then $H_\Omega$ has compact resolvent.
 Furthermore it is well-known that even if $\Omega$
has infinite volume $H_\Omega$ may have compact resolvent.
Necessary and sufficient criteria for $H_\Omega$ to have compact
resolvent in terms of $\Omega$ have been obtained in Section
15.7.3 of \cite{VGM1} and in \cite{VGM2}. For concrete estimates
for the counting function we refer to \cite{vdB} and the
references therein. We also refer to the elementary
Example~\ref{2.14}(1) below.

Note that, for $\dparam>0$, the {\class} ${\cal O}_\dparam$ can be
obtained from the {\class} ${\cal O}_1$ by scaling;
\begin{align*}
   \Omega \in {\cal O}_\dparam \iff \sqrt{\dparam}\,\Omega \in {\cal O}_1,
\end{align*}
for all $\dparam > 0$.  Therefore, we will first derive an estimate
on $\norm{{\Phi}}_1$ for $\Omega$ in the {\class} ${\cal O}_1$.
The general result will then easily follow by scaling.
 In our basic estimate for the {\class} ${\cal O}_1$ we will work with
parameters $r, t$ satisfying
\begin{align*}
    1 \le r < t < \Sigma_\Omega;
\end{align*}
as usual, we let ${\Sigma_\Omega =}\inf\sigmaess(H_\Omega) = \infty$ if
$\sigmaess(H_\Omega) = \emptyset$.

Below, we will derive estimates for all eigenfunctions ${\Phi}$ of $H_\Omega$
associated with eigenvalues $\la_k \in [1,r]$. These estimates will depend
on the number of eigenvalues of $H_\Omega$ in the
interval $[1,t]$, counting multiplicities. Here we use the following
definition: for a self-adjoint operator $T$ and
$t \in \R$ we write
\begin{align*}
   \NN_t(T) := \trace \EE{(-\infty, t]}T,
 \end{align*}
where ${\rm tr}$ denotes the trace, and $\EE IT$ is the spectral
projection of $T$ associated with the interval $I \sse \R$. In
particular, if $T$ is semi-bounded from below and if $t <
\inf\sigmaess(T)$, then $\NN_t(T)$ denotes the number of
eigenvalues of $T$ less than or equal to $t$, counting
multiplicities. If, in our enumeration of eigenvalues, $\la_{k+1}
> \la_k$ for some $k$, we have $N_{\la_k}(T) = k$.

In order to express a certain quantity occurring in the estimate
derived below, we fix (throughout the whole paper) a function
$\rho\in\Cci{\R^d}$, $\rho\ge0$, with $\spt\rho\sse B(0,1/2)$  and
$\int\rho(x)\,dx=1$, and we define
\begin{align}
\label{md-def}
\md:=\max\{1, \norm{\Delta \rho}_1 \}.
\end{align}

The following proposition contains our basic estimate for sets $\Omega \in \OO_1$.
 In the statement we will use, for given $1 \le r < t$, the quantities $\alpha$
and $\beta$ (depending on $r$ and $t$) defined by
 \begin{align*}
    \beta := (t-r)/2
 \end{align*}
and
 \begin{align}
    \alpha := \frac{\min\{\beta, 1\}}{16 \md r},
    \label{(2.7)}
\end{align}
 with $\md$ from \eqref{md-def}.

  \begin{proposition} \label{2.2}
 For any $d\in\N$ there exist constants $C,\cnought > 0$,
 such that for all $\Omega \in \OO_1$,
 for all $1 \le r < t < \inf\sigmaess(H_\Omega)$, and for all
  $n \ge \max\{1,2^{d/2} \cnought / \sqrt\beta\}$ we have
\begin{align}
  \norm{{\Phi}}_1 \le  C\biggl(\! n^{d/2} \sqrt{\NN_t(H_\Omega)}
+  \frac{\sqrt r} \beta
    \Bigl( \frac{n^{2d-2}}\alpha + \frac{n^{d-1}}{\alpha^d} \Bigr)
   \ee^{-\alpha n} \NN_t(H_\Omega)\!\biggr)  \norm{{\Phi}}_2 ,
   \label{(2.9)}
\end{align}
 for all eigenfunctions ${\Phi}$ of $H_\Omega$ associated with
an eigenvalue $\la_k(H_\Omega) \in [1,r]$.
\end{proposition}

 The presence of $\NN_t = \NN_t(H_\Omega)$ takes care of
 situations where there is a ``cloud'' of eigenvalues below $t$.
 Examples of dumb-bell type show that at least a factor $\sqrt{\NN_t}$ appears
  to be necessary. We emphasize that the trivial estimate, valid for all sets
 $\Omega$ of bounded volume,  $\norm{{\Phi}}_1 \le |\Omega|^{1/2} \norm{{\Phi}}_2$,
 is often inadequate, and an estimate in terms of
 spectral data seems to be more appropriate and desirable.

 The constant $\cnought$ appearing in the assumptions of Proposition~\ref{2.2}
 depends solely on the IMS-partition of unity $(\Psi_j)_{j \in \Z^d}$
 introduced at the beginning of Section \ref{proof-prop-2.2}.
 The partition of unity can be constructed in such a way
 that the constant $\cnought$ is easy to compute.
 The constant $C$ appearing on the right hand side of \eqref{(2.9)} could be 
 explicitly computed as a function of the dimension $d$.
 A proof of Proposition \ref{2.2} will be given in
 Section~\ref{proof-prop-2.2}.

In the next step we will reduce the number of free parameters
specifying $n$ first.

\begin{theorem} \label{2.3} (Case of $\OO_1$)

\noindent
 For any $d \in \N$ there exists a constant
 $C$ (depending on $d$ only) such that for any $\Omega \in \OO_1$ and any $r, 
t \in [1,
\Sigma_\Omega)$ with $r < t \leq 3r$, we have
\begin{align}
\norm{{\Phi}}_1^2 \le C
\left(\!\left(\!\frac{r^2}{t-r}\!\right)^{\!d} (\log \NN_t)^d
\NN_t
  + r^{-3}\left(\frac{r^2}{t-r}\!\right)^{\!4d}\,  \right) \norm{{\Phi}}_2^2
\label{(2.10)}
\end{align}
for all eigenfunctions ${\Phi}$ of $H_\Omega$ associated with an eigenvalue
$\la_k \in [1, r]$.
\end{theorem}

\begin{proof} With $\eta:=\frac{t-r}{2r}$ we obtain $0<\eta\le1$,
 $t=(1+2\eta)r$ and $\beta=\eta r$. We will apply Proposition \ref{2.2}
 with two different choices of $n$. We will also use the elementary estimate
 $1\le {1/ \alpha} \le 16 \md \frac r\eta$.
\smallskip

\noindent
(1) For $\log \NN_t \le \max\{1,2^{d/2 -1}c\}$, the choice
$n:=\max\{1,2^{d/2}c/\sqrt\beta\}$ yields
\begin{align}
\begin{split}
  \norm{{\Phi}}_1 
& \le C \max\left\{r^{d-1/2}\eta^{-d-1},r^{d/2} \eta^{-(3d+1)/2}\right\}
                \norm{{\Phi}}_2 \cr
  &  \le C r^{d-1/2}  \eta^{-(3d+1)/2}  \norm{{\Phi}}_2.
\end{split}
   \label{(2.11)}
\end{align}
(As a hint for the computation we mention that it is advantageous to write
\[
  \frac{\sqrt r}\beta \left(\frac{n^{2d-2}}\alpha + 
\frac{n^{d-1}}{\alpha^d} \right)
   =   \frac{\sqrt{r}}{\alpha\beta} \left(n^{2d-2} + (n/\alpha)^{d-1} 
         \right)
\]
and to note that $1/(\alpha\beta) \le C/\eta^2$.)
\smallskip

\noindent
(2) For $\log \NN_t \ge \max\{1,2^{d/2 -1}c\}$ we choose
 $n := 2\frac{\log \NN_t}\alpha$, where we note that the
 condition $n \ge \max\{1,2^{d/2}c/\sqrt\beta\}$ is easily verified.
 We then have  $ \ee^{-\alpha n} \NN_t = 1/\NN_t$ so that
\begin{align*}
      (\log{\NN}_t)^k\cdot \ee^{-\alpha n} \cdot \NN_t \le C_k \qquad (k
\in \N),
\end{align*}
where $C_k := \max\{ (\log u)^k u^{-1} \mid u \ge 1 \}$.

In the first term on the right hand side of \eqref{(2.9)} we
simply estimate $n^{d/2} \le C (r/\eta)^{d/2}(\log\NN_t)^{d/2}$.
As for the second term in the right hand side of \eqref{(2.9)} we
first observe that
\begin{align*}
  \frac{n^{2d-2}}\alpha + \frac{n^{d-1}}{\alpha^d}
   = \frac1{\alpha^{2d-1}} \left( 2^{2d-2} (\log\NN_t)^{2d-2}
       + 2^{d-1}  (\log{\NN}_t)^{d-1}  \right)
\end{align*}
so that
\begin{align*}
   \frac{\sqrt r}\beta
    \left( \frac{n^{2d-2}}\alpha + \frac{n^{d-1}}{\alpha^d} \right)
                       \ee^{-\alpha n} {\NN}_t
    \le \frac C{\sqrt r \eta} (r/\eta)^{2d-1}
  \end{align*}
with $C := (16 m_0)^{2d-1}\bigl(2^{2d-2}C_{2d-2} + 2^{d-1}C_{d-1}\bigr)$.
 Now Proposition \ref{2.2} gives
\begin{align}
   \norm{{\Phi}}_1  \le C \left( (r/\eta)^{d/2} (\log\NN_t)^{d/2} \sqrt{\NN_t}
        +  r^{-3/2} (r/\eta)^{2d} \right) \norm{{\Phi}}_2.
  \label{(2.12)}
\end{align}
(3) Note that the estimate \eqref{(2.11)} implies \eqref{(2.12)}. Inserting the
 definition of $\eta$ and taking squares one obtains \eqref{(2.10)}.
\end{proof}
 \smallskip

 We now pass from the {\class} $\OO_1$ to the {\classes} $\OO_\dparam$,
with $\dparam > 0$, by a straightforward scaling argument.

\begin{theorem}\label{2.4} (Case of $\OO_\dparam$)
 \noindent
 For any $d \in \N$ there exists a constant $C$
(depending on $d$ only) such that for any $\dparam > 0$ and
$\Omega \in \OO_\dparam$ the following estimate holds: If $r,t \in
[\Lambda,\Sigma_\Omega)$ satisfy $r < t \le 3r$, we have
\begin{align*}
   \norm{{\Phi}}_1^2 \le
C \dparam^{-d/2} \!\left(\left(\frac{r^2}{\dparam(t-r)}\right)^d
(\log \NN_t)^d{\NN_t}
   +  \left(\frac r\dparam\right)^{-3}
             \left(\frac{r^2}{\dparam(t-r)}\right)^{4d}\right)\!\norm{{\Phi}}_2^2,
\end{align*}
 for all eigenfunctions ${\Phi}$ of $H_\Omega$ associated with an
 eigenvalue $\la_k \in [\dparam,r]$.
\end{theorem}

As will become clear later on, the factor $\frac{r^2}{\dparam(t-r)}$ in the
above theorem should be read as
$\frac{r}{\dparam} \cdot \frac{r}{t-r}$.

\begin{proof}
Let $\tilde\Omega := \sqrt\dparam\Omega$ ($\in \OO_1$).
Then $\inf\sigmaess(H_{\tilde\Omega}) = \frac 1\dparam \Sigma_\Omega$ and to
each eigenvalue $\la_k$ of $H_\Omega$ below $\Sigma_\Omega$ there
corresponds precisely
one eigenvalue $\tilde\la_k$ of $H_{\tilde\Omega}$ below $\frac 1\dparam
\Sigma_\Omega$;
in fact,
\begin{align*}
        \tilde\la_k = \frac 1\dparam \la_k.
\end{align*}
For the associated eigenfunctions of $H_{\tilde\Omega}$ we take
\begin{align*}
 \tilde  {\Phi}(x) := \dparam^{-d/4} {\Phi}(x/{\sqrt\dparam}) \qquad
(x \in
\tilde\Omega),
\end{align*}
so that, in particular,
\begin{align*}
   \norm{\tilde{\Phi}}_{L_2(\tilde\Omega)} =
\norm{{\Phi}}_{L_2(\Omega)}
   \quad \hbox{\rm and} \quad
   \norm{\tilde{\Phi}}_{L_1(\tilde\Omega)}
   = \dparam^{d/4} \norm{{{\Phi}}}_{L_1(\Omega)} .
 \end{align*}
Setting $\tilde r:=r/\dparam$, $\tilde t:=t/\dparam$ and using the
estimate \eqref{(2.10)} of Theorem~\ref{2.3} for $\tilde{\Phi}$ we
obtain
\begin{align*}
    &\norm{{{\Phi}}}_{L_1(\Omega)}  =  \dparam^{-d/4}
\norm{\tilde{\Phi}}_{L_1(\tilde\Omega)}
  \cr
  & \le
C \dparam^{-d/4} \left(\Bigl(\frac{\tilde r^2}{\tilde t-\tilde
r}\Bigr)^{d/2} (\log\NN_{\tilde
t}(H_{\tilde\Omega}))^{d/2}\sqrt{\NN_{\tilde t}(H_{\tilde\Omega})}
  + \tilde r^{-3/2}\Bigl(\frac{\tilde r^2}{\tilde t-\tilde r}\Bigr)^{2d}
\right)
\norm{\tilde{\Phi}}_{L_2(\tilde\Omega)},
\end{align*}
and the desired result follows since $\NN_{\tilde t}(H_{\tilde\Omega}) =
\NN_t(H_{\Omega})$.
\end{proof}
\smallskip

From Theorem \ref{2.4} we immediately get bounds on
$\norm{\Phi}_p$ for any $p \in [1,2]$ as, trivially,
$\norm{\Phi}_p^p \le \norm{\Phi}_1 + \norm{\Phi}_2^2$. A finer
estimate is obtained through the inequality $\norm{\Phi}_p \le
\norm{\Phi}_1^{\frac2p-1}\norm{\Phi}_2^{2(1-\frac1p)}$.

In the special case where $\sigmaess(H_\Omega) = \emptyset$, we
may take $\Lambda := \la_1$, $r := \la_k$ and $t := (1 + \theta) r
= (1 + \theta) \la_k$, with $0 < \theta \le 1$, in Theorem
\ref{2.4}, which gives the following.

\begin{corollary} \label{2.5}
For any $d \in \N$ there exists a constant $C \ge 0$ such that
the following holds:
 If $\Omega \ne \emptyset$ is an open subset of $\R^d$ with $\sigmaess(H_\Omega) = \emptyset$,
 we have
\begin{align*}
    \norm{{\Phi}}_1^2  \le   C \lambda_1^{-d/2}
\left(\theta^{-d}\Bigl(\frac{\la_k}{\la_1}\Bigr)^d
    (\log\NN_{(1+\theta)\la_k})^d \NN_{(1+\theta)\la_k}
       +  \theta^{-4d}\Bigl(\frac{\la_k}{\la_1}\Bigr)^{4d -3} \right)
\norm{{\Phi}}_2^2,
\end{align*}
for all $0 < \theta\le 1$ and
for all eigenfunctions ${\Phi}$ of $H_\Omega$ associated with the
eigenvalue $\la_k$.
\end{corollary}

The estimate given above will be applied in Section 5 to obtain a
bound for the heat content $Q_\Omega$ in terms of the heat trace
$Z_\Omega$.

In many cases one will be satisfied with
the choice $\theta := 1$, while smaller $\theta$ may be of
interest if $N_t$ is of fast growth. Small $\theta > 0$ are also
important if one is interested in an estimate which depends on the
gap length $\la_{k+1} - \la_k$ (if $\la_{k+1} > \la_k$); choosing
$\theta > 0$ so small that $(1+\theta)\la_k < \la_{k+1}$ we get
$N_{(1+\theta)\la_k} = N_{\la_k} = k$.

\begin{remark} \label{2.13}
 In the special case $d=1$ one can obtain a sharper estimate by direct
calculation, and it is instructive to do that.
Any open set $\Omega \sse \R$
can be written as a countable union of pairwise disjoint open intervals
 $I_k \ne \emptyset$. If one of these intervals has infinite length,
we have $\sigmaess(H_\Omega) = [0,\infty)$ and we thus assume that
all $I_k$ have finite length $\ell_k$. In this case
the operator $H_\Omega$ has pure point spectrum
and there is an orthonormal basis of
eigenfunctions, each having support $I_k$ for some $k$.
Furthermore, $\inf\sigma(H_\Omega)=\pi^2\inf_k 1/\ell_k^2$ and
$\inf\sigmaess(H_\Omega)=\pi^2\liminf_{k\to\infty} 1/\ell_k^2$.
Assume that $\inf\sigma(H_\Omega)<\inf\sigmaess(H_\Omega)$, and let
$\la\in[\inf\sigma(H_\Omega),\inf\sigmaess(H_\Omega))$ be an eigenvalue of
$H_\Omega$. Then there is a finite
subset of the intervals $I_k$, which we may denote as $I_1, \ldots, I_K$ for
simplicity, with the property that $\la$ is an eigenvalue of $H_{I_k}$.
This means that for each $k =1, \ldots , K$ there is a number $j_k \in \N$ such that
\begin{align}\label{(*)}
    \la = \frac{\pi^2 j_k^2}{\ell_k^2};
   \end{align}
the associated normalized eigenfunction of $H_{I_k}$ is
given by
\[
    \phi_k(x) := \sqrt{2/\ell_k}\sin{\frac{\pi j_k}{\ell_k} (x - a_k)},
\]
if $I_k = (a_k,b_k)$. Here $\norm{\phi_k}_1 = \frac{2 \sqrt2} {\pi} \sqrt{\ell_k}$.

Any eigenfunction $\Phi$ of $L_2$-norm 1 of $H_\Omega$ associated
with the eigenvalue $\la$ can be written as $\Phi = \sum_{k=1}^K
\alpha_k \phi_k$ with $\alpha_k \in \C$ and $\sum_{k=1}^K
|\alpha_k|^2 = 1$. As for the $L_1$-norm of $\Phi$, we now
estimate
\[
\norm{\Phi}_1 = \sum_k |\alpha_k| \norm{\phi_k}_1 =  \frac{2\sqrt2}{\pi}
\sum_k |\alpha_k|
   \sqrt{\ell_k} \le  \frac{2 \sqrt2}{\pi} \left(\sum_k \ell_k \right)^{1/2},
\]
by the Schwarz inequality. From \eqref{(*)} we get
\[
   \ell_k = \frac{\pi j_k}{\sqrt{\la}} = \frac{\pi}{\sqrt\la} N_\la(H_{I_k}),
\]
whence $\sum_k \ell_k \le \frac{\pi}{\sqrt\la} N_\la(H_\Omega)$. This leads to the
estimate
\[
\norm{\Phi}_1^2 \le C \la^{-1/2} {N_\la(H_\Omega)} \norm{\Phi}_2^2,
\]
with $C=\frac{8}{\pi}$.

 Comparing this last estimate with Theorem~\ref{Theorem 0.1},
 we see that the leading power $-d/2$ of the eigenvalue and a factor $N_\la(H_\Omega)$
 are there. However, the one-dimensional estimate above is stronger than the estimate given
 in Theorem~\ref{Theorem 0.1}, which was to be expected because there is
 no coupling between the $H_{I_k}$.
\end{remark}

The following theorem gives an upper bound on the $L_\infty$-norm
and a lower bound for the $L_1$-norm of the eigenfunctions of the
Dirichlet
  Laplacian. It is a direct consequence
 of the domain monotonicity of the heat kernel (\cite{PS},
\cite[Theorem 2.1.6]{dav}, \cite[Theorem B.2]{st-vo-96})
and
 corresponding heat kernel bounds. Note that these bounds are valid for
 \emph{all} eigenvalues and eigenfunctions. We include this well-known material
 chiefly for the sake of completeness.

\begin{theorem} \label{2.12}
Let $\Omega \subseteq \R^d$ be open, and suppose that $\la \in (0,\infty)$
is an eigenvalue of $H_\Omega$. Then
\begin{align}\label{(1)}
   \norm{{\Phi}}_\infty \le \left(\frac\ee{2\pi d}\right)^{d/4}
\la^{d/4} \norm{{\Phi}}_2,
\end{align}
and
\begin{align}\label{(2)}
    \norm{{\Phi}}_1 \ge \left(\frac{2\pi d}\ee\right)^{d/4} \la^{-d/4}
\norm{{\Phi}}_2,
\end{align}
for any eigenfunction $\Phi$ of $H_\Omega$ associated with the eigenvalue $\la$.
\end{theorem}

We defer the proof to the end of Section 5 where we will work with heat kernel estimates
anyway.

 \vskip1em

We next look at the case where $\sigmaess(H_\Omega) \ne
\emptyset$. Choosing $\Lambda \le r \in [\Sigma_\Omega/4,
\Sigma_\Omega)$
and letting $t :=(r+2\Sigma_\Omega)/3$ in Theorem~\ref{2.4},
 we obtain estimates that display the
dependence on the distance between $\la_k$ and $\Sigma_\Omega$.

\begin{corollary} \label{2.6}
For any $d \in \N$ there exists a constant $C \ge 0$ such that the
following holds: If $\Omega \ne \emptyset$ is an open set in
$\R^d$ with $\sigmaess(H_\Omega) \ne \emptyset$, and if $r \in
[\max\{ \Lambda,\Sigma_\Omega/4\},\Sigma_\Omega)$, we have,
writing also $t_r := (r+2\Sigma_\Omega)/3$,
\begin{align*}
   \norm{{\Phi}}_1^2
&\le  C  \Lambda^{-d/2} \left(\left(\frac{\Sigma_\Omega^2}
     { \Lambda(\Sigma_\Omega - r)}\right)^{d} (\log\NN_{t_r})^d
          {\NN_{t_r}} \right. \cr
  & \qquad \qquad  + \left. \left(\frac{\Sigma_\Omega} \Lambda\right)^{-3}
   \left(\frac{\Sigma_\Omega^2}\Lambda(\Sigma_\Omega - r)\right)^{4d}
                \right)\norm{\Phi_{k}}_2^2,
 \end{align*}
for all eigenfunctions ${\Phi}$ of $H_\Omega$ associated with an
eigenvalue $\la_k \in [ \Lambda,r]$.
\end{corollary}

The above estimate is mainly of interest for eigenvalues $\la_k$
close to $\Sigma_\Omega$. For eigenvalues $\la_k$ close to
$\Lambda$, a better, but also more complicated, estimate would be
obtained by choosing $r \in [\Lambda, \Sigma_\Omega)$ and $t :=
\min\{3r , (r + 2\Sigma_\Omega)/3 \}$.

 The following examples illustrate various points made in the preceding text.

 \begin{examples} \label{2.14}
\vskip.5ex
 (1) There are domains $\Omega \sse \R^d$ such that $H_\Omega$ has
compact resolvent while $\R^d \setminus \Omega$ has measure zero. In fact,
 consider a sequence of pairwise disjoint open cubes $Q_k \sse \R^d$,
$k\in\N$,
 enjoying the properties
{\parindent=3em

 $(i)$ $\diam (Q_k) \to 0$, as $k \to \infty$;

 $(ii)$  $\bigcup_{k\in\N}\overline{Q}_k = \R^d$.
}

 Then the Dirichlet Laplacian of $\Omega := \bigcup_{k\in\N} Q_k$ has compact
resolvent.
 In addition to properties $(i)$ and $(ii)$ one may require that any compact subset
$K \sse \R^d$ meets only finitely many of the $Q_k$.

  To obtain a connected $\Omega'$ from the above $\Omega$ it is
enough to open small ``doors'' in the surfaces that separate the
 cubes.

  (2) Here we discuss examples of eigenfunctions which are not in $L_1$.
 Let $\Omega = \bigcup_1^\infty I_k \sse \R$ be the disjoint union of
open intervals of length $1$, $\phi_0$ the normalized
eigenfunction of $H_{(0,1)}$ to the lowest eigenvalue $\la_0$.
Then $\la_0\in\sigmaess(H_\Omega)$. Let $\phi_k$ be the translate
of $\phi_0$ to $I_k$, and let $\alpha \in \ell_2\setminus\ell_1$.
Then
\begin{align*}
\phi:=\sum_k \alpha_k \phi_k \in L_2(\Omega)
\end{align*}
is an eigenfunction of $H_\Omega$ to the eigenvalue $\la_0$, but
$\phi \notin L_1(\Omega)$.

 It is easy to generalize this idea to
 higher dimensions.
 Finding examples of \emph{domains} $\Omega$ in $\R^d$, for $d \ge 2$, with the
 property that $H_\Omega$ has an eigenfunction which is not in $L_1$ seems
 to be much harder. One might think of a quantum wave guide
 perturbed in such a way that an eigenvalue is
 generated right at a boundary point of the essential spectrum.

 For the sake of comparison we note that there are examples of
 Schr\"odinger eigenfunctions on $(0,\infty)$ which are not in $L_1$
 (see \cite{east-kalf-82}); the associated eigenvalues belong
 to the essential spectrum.

(3) It is illuminating to compare the situation of $m$ disjoint
balls with the case where $m$ balls are connected by thin
passages, as in a dumb-bell domain for $m=2$. Here one can see
several aspects of the presence of $N_{2\la_k}(H_\Omega)$ in our
estimates.

 We begin with a (disconnected) open
set $\Omega_m \subseteq \R^d$ consisting of $m$ pairwise disjoint
open balls of radius $1$, say. Let $\la_1$ denote the lowest
eigenvalue of the Dirichlet Laplacian on such a ball. Then the
lowest eigenvalue of $H_{\Omega_m}$ is $\la_1$ and the associated
eigenspace has dimension $m$. It is easy to see that there
 is an eigenfunction $\Phi$ of $H_{\Omega_m}$ to the eigenvalue $\la_1$
with the properties $\norm{\Phi}_2 = 1$ and $\norm{\Phi}_1^2 = m$.
Here an estimate involving  $N_{\la_1}(H_{\Omega_m})$ (instead of
 $N_{2\la_1}(H_{\Omega_m})$) would be possible.

We now generalize domains of dumb-bell type. For $2 \le m \in \N$,
we place
$m$ balls of radius $1$ at the corners of a regular $m$-gon with edge
length $3$,  and connect each of these balls with its two neighbors
by narrow passages of width $0 < \eps \le 1$ along the edges. Call
these domains ${\Omega_{m,\eps}}$. By Perron-Frobenius theory, the
ground state eigenvalue $\lambda_{1; m, \eps}$ of
${\Omega_{m,\eps}}$ is simple and the associated eigenfunction
$\Phi_{1; m,\eps}$ can be chosen strictly positive; furthermore,
$\Phi_{1;m ,\eps}$ is invariant under rotation of the corners. Let
$\norm{\Phi_{1;m,\eps}}_2 = 1$. As $\eps \downarrow 0$, monotone
convergence of quadratic forms, combined with compactness, implies
that $\norm{\Phi_{1;m,\eps}}_1^2 \to m$.

In the disconnected case, we have $N_{\la_1}(H_\Omega) = m$, while
$N_{\la_1}(H_{\Omega_{m,\eps}})$ is equal to one in the connected
case. In the connected case, however, there is a cluster of $m$
eigenvalues close to $\la_1$ (for $\eps > 0$ small), and
$N_{2\la_1}(H_{\Omega_{m,\eps}}) \to m$ as $\eps \downarrow 0$.
Therefore, in the connected case the estimate should better
contain a factor like $N_{t}(H_{\Omega_{m,\eps}})$, with suitable
$t>\la_1$.

Note that, in both cases, the lower bound of
Theorem~\ref{2.12} does not capture the above behavior since it
provides a constant which is independent of $m$.

  (4) Examples of open sets $\Omega \subseteq \R^d$ with discrete eigenvalues
 located in a gap of the essential spectrum can be obtained by suitable
  perturbations of periodic quantum wave-guides. Consider
 open, connected, periodic sets $\Omega_0 \sse \R^2$ of the form
 \[
 \Omega_0 = \{(x,y)\in\R^2 \mid f_1(x) < y < f_2(x) \}
\]
  where $f_1$ and $f_2$ are smooth periodic functions of the same period
 satisfying  $f_1(x) < f_2(x)$. The spectrum of
  $H_{\Omega_0}$ is pure essential spectrum. In this class it is easy to
 find domains with a spectral gap.
  The simplest examples are obtained by joining discs $B((k,0),1/4)$
 ($k \in \Z$) by narrow passages along the $x$-axis, but there are also
  more demanding examples like the ones studied by Yoshitomi \cite{yo-98}.
   Local perturbations of
  the boundary of $\Omega_0$ may produce discrete eigenvalues below the essential
  spectrum, but also discrete eigenvalues inside a given gap of the essential
 spectrum. See for example \cite{Post}.

(5)  We finally discuss a class of examples which
are closely related to Remark \ref{2.13}.
Let $\Omega_0 \subseteq \R^d$ be open and bounded.
Let $(\ell_k)_{k\in\N}$ be a sequence in $(0,\infty)$, and let $\Omega$ be
the disjoint union of a sequence  $(\Omega_k)$, where $\Omega_k$ is a
translate of $\ell_k\Omega_0$, for all $k \in \N$. 
For a dilation
$\ell \Omega_0$, with $\ell > 0$, it follows from the lower bound given
in \eqref{e4} below that
\begin{align*} 
   N_\la(H_{\ell \Omega_0}) = N_{\ell^2 \la}(H_{\Omega_0}) \ge c_0 \ell^d \la^{d/2}
\end{align*}
for all eigenvalues $\la$ of $H_{\ell \Omega_0}$, with a positive 
constant $c_0$.

Assume that
$\inf\sigma(H_\Omega) < \inf\sigmaess(H_\Omega)$,
and let
$\la \in [\inf\sigma(H_\Omega),\inf\sigmaess(H_\Omega))$ be an 
eigenvalue of $H_\Omega$, with associated eigenfunction $\Phi$. 
Arguing as in Remark~\ref{2.13}, one then obtains that
\[
\|\Phi\|_1^2
\le\frac{\vol(\Omega_0)}{c_0}\lambda^{-d/2}N_\lambda(H_\Omega)\|\Phi\|_2^2.
\]

 \end{examples}

For completeness we include here a simple lower bound
for the eigenvalue counting function $N_\la(H_\Omega)$ which does not invoke
Weyl's Theorem and which comes  with an explicit constant.

Let $\Omega$ be an open set in $\R^d$ and suppose that 
$\sigmaess(H_\Omega)=\emptyset$.
Let $\CC(\Omega)$ be the collection of open cubes
contained in $\Omega$ and define
\begin{equation*}
\gamma(\Omega):=\sup\{\vol(A)\mid A\in \CC(\Omega)\}.
\end{equation*}
Let $A\in\CC(\Omega)$, $\vol(A)=a^d$.
By domain monotonicity of the Dirichlet eigenvalues we have that
\begin{align*}
\begin{split}
N_t(H_\Omega)&\ge N_t(H_{A}) 
=\bigl|\{(k_1,\ldots,k_d)\in\N^d\mid\pi^2(k_1^2+\cdots+k_d^2)\le t
a^2\}\bigr| \\ &\ge \bigl|\{k\in \N\mid d\pi^2k^2 \le t
a^2\}\bigr|^d  
\ge \left[\frac{at^{1/2}}{\pi
d^{1/2}} \right]^d.
\end{split}
\end{align*} 
Since $\max\{[x],1\}\ge x/2,$ we conclude that
\begin{equation*}
N_t(H_\Omega)\ge\left(\frac{at^{1/2}}{2\pi d^{1/2}} \right)^d=
(2\pi)^{-d}d^{-d/2}\vol(A)t^{d/2}\qquad (t\ge \lambda_1).
\end{equation*}
Taking the supremum over all $A\in\CC(\Omega)$ we finally obtain
the lower bound
\begin{equation}\label{e4}
N_t(H_\Omega)\ge
(2\pi)^{-d}d^{-d/2}\gamma(\Omega)t^{d/2}\qquad (t\ge \lambda_1).
\end{equation}

\section{Proof of Proposition \ref{2.2}}

\label{proof-prop-2.2}

 We define coverings of $\R^d$ by cubes $Q_{n,j}$ ($j \in \Z^d$), for
 $n \in \N$, and subordinate IMS-partitions of unity $(\Psi_{n,j})_{j\in\Z^d}$.
 Let $Q_0 := (-1,1)^d$ denote the standard cube of side length 2 centered at
 the point $0 \in \R^d$, and let $Q_j := Q_0 + j$, for $j \in \Z^d$, denote the
translates of $Q_0$. Pick some non-negative function
$\psi \in C_c^\infty(Q_0)$, with
$\psi(x) \ge 1$ for all $x\in\frac12 Q_0$.
Let $\psi_j \in C_c^\infty(Q_j)$ be defined by $\psi_j(x) := \psi_0(x
-j)$. Extending the $\psi_j$ by zero to all of $\R^d$, we note
that the function
\begin{align*}
   w := \sum_{j \in \Z^d} \psi_j^2
\end{align*}
is periodic and positive. We now define the
IMS-partition of unity $(\Psi_j)_{j\in\Z^d}$
(\cite{cy-fr-ki-si}) by
\begin{align*}
   \Psi_j := \frac{\psi_j}{\sqrt w} \qquad (j \in \Z^d),
\end{align*}
so that $\sum_{j \in \Z^d} \Psi_j^2(x) = 1$ for all $x \in \R^d$. (Notice the
square; this is {\it not} a standard partition of unity!)
Obviously $\spt\Psi_j \sse Q_j$ for all $j \in \Z^d$. Furthermore,
$\Psi_j$ is a translate of $\Psi_0$, and thus
\begin{align}
\cnought:=   \norm{\nabla \Psi_0}_\infty=\norm{\nabla \Psi_j}_\infty \qquad
(j \in \Z^d).
    \label{(3.3)}
\end{align}
(It would be easy to indicate an upper bound for $\cnought$
in terms of $\norm{\nabla \psi}_\infty$.)

We finally produce scaled versions defined as
\begin{align*}
\Psi_{n,j} := \Psi_j\bigl(\tfrac\cdot n) \qquad (n \in \N,
\ j\in \Z^d);
\end{align*}
notice that $\Psi_{n,j}$ has support in the cube $Q_{n,j} := n Q_j = nQ_0 + nj$.
Then
\begin{align*}
    \sum_{j \in \Z^d} \Psi^2_{n,j}(x) = 1 \qquad (x \in \R^d),
\end{align*}
and
\begin{align}
     \norm{\nabla \Psi_{n,j}}_\infty = \cnought/n \qquad (n \in \N,\ j \in
\Z^d).
     \label{(3.6)}
\end{align}

 For $1 < t < \inf\sigmaess(H_\Omega)$ and $n\in\N$, we now let
\begin{align*}
        \JJ(n,t) := \{ j \in \Z^d \mid \la_1(H_{\Omega \cap
   Q_{n,j}}) < t \}.
\end{align*}
 We then define
\begin{align*}
    F_n := \bigcup_{j \in \JJ(n,t)} Q_{n,j}. 
\end{align*}
For later use we also introduce
\begin{align*}
  {\tilde Q}_0 := 2\,Q_0 = (-2,2)^d, \qquad {\tilde Q}_{n,j} := n {\tilde
Q}_0 + nj,
   \qquad {\tilde F}_n := \bigcup_{j \in \JJ(n,t)} {\tilde Q}_{n,j},
\end{align*}
\begin{align*}
   {\tilde{\tilde Q}}_0 := 3\,Q_0 = (-3,3)^d, \qquad {\tilde{\tilde Q}}_{n,j}
    := n    {\tilde{\tilde Q}}_0 + nj,
    \qquad {\tilde{\tilde F}}_n := \bigcup_{j \in \JJ(n,t)} {\tilde{\tilde
Q}}_{n,j}.
\end{align*}
We then have ${F}_n \sse  {\tilde{F}}_n \sse  {\tilde{\tilde{F}}}_n$, and
\begin{align*}
     \dist({F}_n, \partial {\tilde{F}}_n) \ge n,
    \qquad \dist( {\tilde{F}}_n, \partial {\tilde{\tilde{F}}}_n) \ge n,
\end{align*}
for all $n\in \N$.

The following lemma shows that there is only a finite number of ``cells''
$Q_{n,j}$ such that the
infimum of the spectrum of the Dirichlet Laplacian of $\Omega \cap Q_{n,j}$ is
smaller than $t$.

\begin{lemma} \label{3.1}
Let $1 < t < \inf\sigmaess(H_\Omega)$ and let $n\in\N$.
Then $\JJ(n,t)$ is finite and
\begin{align*}
    |\JJ(n,t)| \le 3^d\, {\NN}_t(H_\Omega),
\end{align*}
where $|\JJ(n,t)|$ denotes the number of elements in $\JJ(n,t)$.
\end{lemma}

\begin{proof}
Let $\JJ\sse\JJ(n,t)$ be finite. There exists $\JJ'\sse\JJ$ such that the
cubes $(Q_{n,j})_{j\in\JJ'}$ are pairwise disjoint and
$n\JJ\sse\bigcup_{j\in\JJ'}\overline{Q_{n,j}}$.
The latter property implies that $|\JJ|\le 3^d|\JJ'|$.

For each of the cubes $Q_{n,j}$ ($j\in\JJ'$) there exists a function
$\phi_j \in
C_c^\infty(\Omega \cap Q_{n,j})$  with $\norm{\phi_j} = 1$ and
$\norm{\nabla \phi_{j}}^2 < t$.
 Then the min-max principle, applied to the subspace spanned by the set
 $\{\phi_j;\, j\in\JJ'\}$, implies that $|\JJ'|\le{\NN}_t(H_\Omega)$.
 The assertions follow from these two inequalities.
\end{proof}

Below we obtain a lower bound for the spectrum of the Dirichlet
Laplacian in $G_n := \Omega \setminus \overline {F_n}$.
\begin{lemma}\label{3.2}
 Let $1 < s < t < \inf\sigmaess(H_\Omega)$ and $n \ge
 n_0:=2^{d/2}\cnought/\sqrt{t-s}$,  with $\cnought$ from \eqref{(3.3)}.
 Let $H_{G_n}$ denote the Dirichlet Laplacian of
 $G_n := \Omega \setminus \overline {F_n}$. Then
\begin{align*}
   \inf \sigma(H_{G_n}) \ge s.
\end{align*}
\end{lemma}

\begin{proof}
 By the IMS-localization formula \cite[Theorem 3.2]{cy-fr-ki-si} and
 \eqref{(3.6)} we have for any
 $\phi \in C_c^\infty(G_n)$
\begin{align*}
    \scapro{H_{G_n} \phi}{\phi} &= \sum_{j\in \Z^d} \scapro{H_{G_n}
   \Psi_{n,j}\phi}{\Psi_{n,j}\phi} - \int \sum_{j\in\Z^d} |\nabla\Psi_{n,j}|^2
|\phi|^2 \di x \cr
   &\ge t \sum_{j\in\Z^d} \norm{\Psi_{n,j}\phi}^2 - 2^d \frac{\cnought^2} {n^2}
\norm{\phi}^2 \cr
   & = (t -  \frac{2^d \cnought^2}{n^2}) \norm{\phi}^2 \ge s \norm{\phi}^2.
\end{align*}
In the estimate we have used that $\Psi_{n,j}\phi=0$ for $j\in\JJ(n,t)$.
The factor $2^d$ takes into account the fact that at most $2^d$ functions
$\Psi_{n,j}$ can be simultaneously non-zero at any given point $x$.
\end{proof}
\smallskip

We next consider a smoothed version of the
indicator
function of ${\tilde F}_n$, defined as
\begin{align}
   \xi_n := \rho * \indic_{{\tilde F}_n},
   \label{(3.15)}
\end{align}
with $\rho$ defined in Section~\ref{estimates}.
For $n \in \N$ we have
$\supp{\xi_n} \sse  {\tilde{\tilde{F}}}_n$ and
${\xi_n}(x) = 1$ for $x \in {F}_n$.
Furthermore, $0 \le {\xi_n}(x) \le 1$ and
$\|\nabla {\xi_n}(x)\|_\infty \le C$,
$\|\Delta {\xi_n}(x)\|_\infty \le C$
 for some
constant $C \ge 0$ which is independent of $n$ and $\Omega$. Also,
 \[
\spt\nabla\xi_n \sse \{x\in\R^d\mid \dist(x,\partial\tilde F_n)<1/2\}.
\]

It will be convenient to cover the support of $\nabla\xi_n$
by (non-overlapping) cubes of side length $1$, given by
\begin{align}
   {\check Q}_0 := (-1/2,1/2]^d, \quad  {\check Q}_\ell :=  {\check Q}_0 +
   \ell \qquad (\ell \in \Z^d),
    \label{(3.16)}
\end{align}
and we will write  ${{\check\chi}_\ell} := \indic_{{\check
Q}_\ell}$.

We then let
\begin{align}
  Z_n :=  \Z^d \cap \partial{\tilde F}_n.
   \label{(3.17)}
\end{align}
Note that this implies $\supp \nabla\xi_n \sse \bigcup_{\ell \in
Z_n} {\check Q}_\ell$.  We then have
\begin{align}
  |Z_n| 
\le |\JJ(n,t)| \, |\Z^d\cap\partial\tilde Q_{n,0}|
\le C \, n^{d-1} \, \NN_t(H_\Omega),
    \label{(3.19)}
\end{align}
by Lemma \ref{3.1}, with $C=2^{ 3d-2}\,3^d$. We furthermore let
\begin{align*}
  Y_n := \{ j \in \Z^d \mid {\check Q}_j \cap (\Omega \setminus {\tilde{\tilde
  F}}_n) \ne \emptyset \},
\end{align*}
so that $\Omega \setminus {\tilde{\tilde  F}}_n \sse \bigcup_{j \in Y_n}
{\check Q}_j$.

 We now quantify the exponential decay of ${\Phi}$ as  we move away from
 the set ${\tilde{F}}_n$.

\begin{lemma} \label{3.3}
There exists a constant $C \ge 0$ with the following property.
If $1 \le r < t < \inf\sigmaess(H_\Omega)$, if $n \ge n_0$ (with $n_0$
from Lemma~\ref{3.2}), and if ${\xi_n}$ from \eqref{(3.15)}, then
 \begin{align*}
  \norm{ {\check\chi}_j (1-{\xi_n}) {\Phi}}_1
     \le C \frac{\sqrt{r} }{t-r} \sum_{\ell \in Z_n} \ee^{-\alpha |j -
  \ell|}  \qquad (j \in \Z^d)
 \end{align*}
(with $\alpha$ from \eqref{(2.7)}), for all normalized
eigenfunctions ${\Phi}$ associated with an eigenvalue $\la_k \in
[1,r]$.
\end{lemma}

We defer the proof of Lemma \ref{3.3} to Section~\ref{exp-dec-est}.

\begin{proof}[Proof of Proposition~\ref{2.2}]

First we treat the case where $n\in\N$, starting from 
estimate
\begin{align}
   \norm{{{\Phi}}}_1  = \int_{\Omega \cap {\tilde{\tilde{F}}}_n} |{{\Phi}}(x)|
\di x
     + \int_{\Omega \setminus  {\tilde{\tilde{F}}}_n} |{{\Phi}}(x)| \di x
   =: I_{n,1} + I_{n,2}.
   \label{(3.23)}
\end{align}
By the Schwarz inequality and Lemma \ref{3.1}, the first term on the right hand
side of \eqref{(3.23)}
can be estimated as follows:
\begin{align}
    I_{n,1} \le \vol_d \bigl({\tilde{\tilde{F}}}_n\bigr) ^{1/2}
\norm{{{\Phi}}}_2
    \le |\JJ(n,t)|^{1/2} (6n)^{d/2}  \le 3^d\,2^{d/2}\, n^{d/2}
\sqrt{{\NN}_t(H_\Omega)}.
\end{align}

As for the second term on the right hand side of \eqref{(3.23)},
we note that for $j \in Y_n$ and $\ell \in Z_n$ we have
\begin{align*}
    |j - \ell| \ge n - 1.
\end{align*}
 It now follows by Lemma \ref{3.3} that
\begin{align*}
    I_{n,2}
    &  \leq \sum_{j \in Y_n} \norm{{\check \chi}_j (1-{\xi_n}) {{\Phi}}}_1
     \le C  \frac{\sqrt{r}}{t-r} \sum_{j \in Y_n} \sum_{\ell \in Z_n}
\ee^{-\alpha
                     |j-\ell|} \cr
    & = C  \frac{\sqrt{r} }{t-r} \sum_{\ell \in Z_n} \sum_{j \in Y_n}
\ee^{-\alpha
                     |j-\ell|}
        \le  C \frac{\sqrt{r}}{t-r} \cdot \bigl(\sup_{\ell\in Z_n}
M(n,\ell)\bigr) \cdot
          |Z_n|,
\end{align*}
where $M(n,\ell) :=  \sum_{j \in Y_n} \ee^{-\alpha |j-\ell|}
$ for $\ell \in
Z_n$.
  Here $|Z_n| \le C n^{d-1} \, {\NN}_t(H_\Omega)$ by \eqref{(3.19)}, and
\begin{align*}   
M(n,\ell) 
&\le \sum_{j\in\Z^d,\ |j|\ge n-1}\ee^{-\alpha|j|}
\le 
\ee^{\alpha\sqrt d}\int_{\{\xi \in \R^d; |\xi| \ge n - 1 \}} 
\ee^{-\alpha |\xi|} \di\xi\\
    &\le C \left(\frac{n^{d-1}}{\alpha} + \frac 1 {\alpha^d} \right)
\ee^{-\alpha
n},
\end{align*}
for all $n \in \N$ and $\ell \in Z_n$. We therefore obtain the estimate
\begin{align}
   I_{n,2}  \le C  \frac{\sqrt r}{t-r}\, n^{d-1} \,
    {\NN}_t(H_\Omega) \left( \frac{n^{d-1}}\alpha + \frac1{\alpha^d} \right)
\ee^{-\alpha n}.
\end{align}
Now (1.3) follows from (2.8) and (2.9).

Finally, we reduce the case of non-integer $n$ to the case treated above. If
$n\in(0,\infty)$ satisfies the required inequality, then $\tilde n:=\lceil
n\rceil$ (the smallest integer $\ge n$) belongs to $\N$, and the
asserted inequality holds for $n$ replaced with $\tilde n$. Readjusting the
constant $C$, one then obtains the estimate with $n$.
\end{proof}

\section{Exponential decay estimates}
\label{exp-dec-est}

This section is devoted to quantitative exponential decay
estimates and a proof of Lemma \ref{3.3}. We use the method of
\emph{boosting}, a well-known tool in the study of eigenfunctions
of Schr\"odinger operators (cf.\ \cite[p.\,37]{his-sig} for a
survey of the literature). Here the operator is sandwiched between
$\ee^{\gamma \cdot x}$ and $\ee^{-\gamma \cdot x}$ for $\gamma \in
\R^d$. This method dates from the eighties and yields exponential
decay in the $L_2$-sense. Below, we follow to some extent the
proof of \cite[Lemma 6]{dei-hem}. To keep technicalities as simple
as possible we will not work with $\ee^{\pm\gamma \cdot x}$ but
with smoothed cut-offs of these functions.

We first consider general real-valued functions $f \in C^\infty(\R^d)$ with
$f$, $\nabla f$ and $\Delta f$ bounded; only
later on we will specify $f$ to coincide with $\gamma \cdot (x - k)$ on a large
ball.
 Let $G \sse \R^d$ open. By \cite[Theorem~VI-2.1]{kat}, it
is then easy to see that
$fu \in D(H_G)$ for all $u \in D(H_G)$; furthermore,
\begin{align*}
  \ee^{-f} H_G \ee^{f}
    = H_G - 2 \nabla f \cdot \nabla - \Delta f - |\nabla f|^2.
\end{align*}
Here we note that the perturbation $ - 2 \nabla f \cdot \nabla - \Delta f -
|\nabla f|^2$
has relative form-bound zero with respect to $H_G$ on $H_0^1(G)$.
We let $H_{G,f}$ denote the (unique) m-sectorial closed operator associated
with $H_G -  2 \nabla f \cdot \nabla - \Delta f - |\nabla f|^2$
by \cite[Theorem~VI-3.4 or 3.9]{kat}. Also, using
\cite[Theorem~VI-2.1(iii)]{kat}, one can easily see that
$D(H_{G,f}) =  D(H_G)$.
From the inequalities
\begin{align}
\norm{\nabla \phi}_2^2  = h_G[\phi, \phi] =  \scapro{H_G\phi}{\phi}
  \le \frac{\delta^2 } 2 \norm{H_G\phi}_2^2 +  \frac 1 {2\delta^2}
  \norm{\phi}_2^2,
   \label{(A.2)}
\end{align}
valid for all $\phi \in D(H_G)$ and all $\delta > 0$, we see that
$H_{G,f} - H_G$ is relatively bounded with respect to $H_G$ in the operator
sense, with relative bound zero.

\begin{lemma} \label{A.1}
Let $G \sse \R^d$ be open and let $H_G$, the Dirichlet
Laplacian of $G$, be such that $\sigma_0 := \inf\sigma(H_G) > 1$. For
 $1 \le r < s  < \sigma_0$ and $m \ge 1$ define
\begin{align}
   \alpha := \frac{\min \{s-r, 1\}}{16 m r}.
   \label{(A.3)}
\end{align}
Let $f \in C^\infty(\R^d; \R)$ be bounded with
$\|\nabla f\|_\infty \le m$ and $\|\Delta f\|_\infty \le m$.

We then have $ [1,r]
\sse
\varrho(H_{G,\alpha f})$,
and
\begin{align*}
   \norm{(H_{G, \alpha f} - \la)^{-1}}
    \le \frac 2 {s - r},
\end{align*}
for all $\la \in [1,r]$. Furthermore, for the same $\la$, one has
\begin{align*}
  ( H_{G, \alpha f}  - \la)^{-1} = \ee^{-\alpha f} (H_G - \la)^{-1} \ee^{\alpha
f}.
\end{align*}
\end{lemma}

\begin{proof}
We are going to apply \cite[Theorem IV-1.16]{kat} to $T := H_G - \la $ and
 $S := H_{G,\alpha f} - \la$. In estimating the term containing $\nabla f \cdot
\nabla$
we use \eqref{(A.2)} with $\delta := 1$ so that
\begin{align*}
  2\norm{\nabla f \cdot \nabla \phi}_2 \le \sqrt2 m \norm{(H_G - \la) \phi}_2
    + \sqrt2 m (\la + 1) \norm{\phi}_2.
\end{align*}
It is not difficult to
see that the numbers $a := \sqrt{2} m |\alpha| (2 + \la) + |\alpha|^2m^2$  and
 $b := \sqrt{2} |\alpha|m$ satisfy the condition
 $a \norm{(H_G - \la)^{-1}} + b \le 1/2$.
 We thus see that any $\la \in [1, r]$  belongs to the resolvent set of
$H_{G,\alpha f}$; furthermore, \cite[eqation~IV-(1.31)]{kat} yields the
estimate
$\norm{( H_{G,\alpha f} - \la)^{-1}} \le \frac2{s - r}$.
Direct computation shows that $ \ee^{-\alpha f} (H_G - \la)^{-1} \ee^{\alpha f}$
is the inverse of $H_{G,\alpha f} - \la$.
\end{proof}
\smallskip

For the application of Lemma \ref{A.1} we need to construct specific
functions~$f$. Consider first
\begin{align*}
    \phi_{k, \ell}(x) := \frac 1{|\ell - k|} \scapro{x - k}{\ell - k}
    \qquad (x\in\R^d,\ k,\ell \in \Z^d,\ k \ne \ell),
\end{align*}
so that $ {\phi}_{k, \ell}(k) = 0$ and $ {\phi}_{k, \ell}(\ell) = |\ell -
k|$.  We next take, for $R \ge 1$,
\begin{align*}
    f_{R, k, \ell} :=
\rho_R
* \left((\phi_{k,\ell} \wedge R) \vee
    (-R)\right),
\end{align*}
with $\rho$ defined in Section~\ref{estimates} and
$\rho_R:=\frac1{R^d}\rho\bigl(\frac\cdot R\bigr)$; recall that
$\spt\rho\sse B(0,1/2)$.
We then have $\|\nabla f_{R,k,\ell}\|_\infty \le 1$
and  $\|\Delta f_{R,k,\ell}\|_\infty \le \frac1R\|\Delta \rho\|_1$.
\begin{lemma}\label{A.2}
Let $G \sse \R^d$ be an open set with $\sigma_0
:= \inf\sigma(H_G) > 1$, and
let $1 \le r < s < \sigma_0$. Let
\begin{align*}
   \alpha := \frac{\min \{ s - r, 1\}} {16 \md r},
\end{align*}
with $\md$ from \eqref{md-def}. Finally,
let ${\check Q}_k$ as in \eqref{(3.16)}, and ${\check\chi}_k := \indic_{{\check
Q}_k}$ for $k \in \Z^d$.

Then there exists a constant $C \ge 0$, depending only on $d$, such that
\begin{align*}
   \norm{{\check\chi}_k (H_G - \la)^{-1} {{\check\chi}_\ell}} \le \frac C{s-r}
\ee^{-\alpha |k - \ell|} \qquad
   (k, \ell \in \Z^d),
\end{align*}
for any $\la \in [1,r]$.
\end{lemma}

\begin{proof}
In the cases where $k = \ell$, the inequality holds with $C=1$.

Let $k,\ell \in \Z^d$, $k\ne\ell$, and let
$R :=  2 |k - \ell|$.
With $f :=  \alpha f_{R, k, \ell}$ and $E_f$ denoting multiplication
by the function $\ee^f$, we compute
\begin{align*}
    \norm{{\check\chi}_k (H_G - \la)^{-1} {{\check\chi}_\ell}}
 &= \norm{{\check\chi}_k E_f E_{-f} (H_G - \la)^{-1}
    E_fE_{-f} {{\check\chi}_\ell}} \cr
 & \le \norm{{\check\chi}_k \ee^f }_\infty \norm{E_{-f} (H_G - \la)^{-1}
    E_f} \norm{\ee^{-f} {{\check\chi}_\ell}}_\infty \cr
 & \le \frac{C_0}{s-r} \norm{\ee^{-f} {{\check\chi}_\ell}}_\infty,
\end{align*}
by Lemma \ref{A.1}. Since $0 < \alphad \le 1$ we have $ \norm{{\check\chi}_k
\ee^f}_\infty \le
\ee^{\sqrt{d}/2}$ and we may thus choose $C_0 := 9\ee^{\sqrt{d}/2}$.
Furthermore,
\begin{align*}
   \norm{\ee^{-f} {{\check\chi}_\ell}}_\infty
   \le \sup_{x \in {\check Q}_\ell} \ee^{-\alphad |x - k|}
   \le \ee^{-\alpha (|\ell - k|  - \sqrt{d}/2)} = C_1 \ee^{-\alphad |k -
\ell|},
\end{align*}
with $C_1 :=  \ee^{\alpha \sqrt{d}/2}$.
\end{proof}

\begin{proof}[Proof of Lemma \ref{3.3}]
In this proof we specify the number $s$, occurring in
Lemma~\ref{3.2} and in Lemma~\ref{A.2}, as
\[
s:=\frac{r+t}2.
\]
With this definition the quantity $\alpha$ occurring in
Lemma~\ref{A.2} becomes $\alpha$ from \eqref{(2.7)}, and $n_0$
from Lemma~\ref{3.2} becomes the lower bound for $n$ in
Proposition~\ref{2.2}.

Let $n\ge n_0$ (from Lemma~\ref{3.2}), and let ${\xi_n} \in C_c^\infty(\R^d)$
be as defined in
\eqref{(3.15)}.  As in Lemma~\ref{3.2} we consider $G = G_n := \Omega
\setminus \overline {F_n}$,
and we conclude from Lemma~\ref{3.2} that $\inf \sigma(H_G) \ge
s$, for
 the Dirichlet Laplacian $H_G$ of $G$.

Let ${\Phi}$ be a normalized eigenfunction of $H_\Omega$,
associated with an eigenvalue $\la_k \in [1,r]$. It is easy to see
that
 $(1 - \xi_n){\Phi}$ belongs to
$D(H_\Omega)\cap D(H_G)$.
The usual calculation yields
\begin{align}
\begin{split}
  (H_G - \la_k) ((1 - {\xi_n}){{{\Phi}}}) &= (H_\Omega -
\la_k)((1-{\xi_n}){{\Phi}})\\
&   = 2 \nabla{\xi_n}\cdot \nabla{{\Phi}} + (\Delta{\xi_n}) {{\Phi}}
   =: \eta.
     \label{(A.15)}
\end{split}
\end{align}
Here $\supp \eta \sse \supp \nabla \xi_n \sse \bigcup_{\ell \in Z_n}
{\check Q}_\ell$,
 with $Z_n$ from \eqref{(3.17)}, and
\begin{align*}
  \norm{\eta}_2 \le C (1 + \norm{\nabla{{\Phi}}}_2)
   =  C (1 + \sqrt{\la_k}),
\end{align*}
with a constant $C \ge 0$ depending only on $\norm{\nabla{\xi_n}}_\infty$ and
$\norm{\Delta{\xi_n}}_\infty$
(and therefore not depending on $n$ and $\Omega$),
and \eqref{(A.15)} yields
\begin{align}
  (1 - {\xi_n}){{\Phi}} = (H_G - \la_k)^{-1} \eta.
   \label{(A.17)}
\end{align}
 We now obtain
\begin{align*}
  {\check\chi}_j (1 - {\xi_n}) {{\Phi}} = {\check\chi}_j (H_G - \la)^{-1}
(\sum_{\ell \in Z_n}{{\check\chi}_\ell} \eta)
\end{align*}
for any $j \in \Z^d$,
so that, by Lemma~\ref{A.2} and \eqref{(A.17)},
\begin{align*}
    \norm{ {\check\chi}_j (1-{\xi_n}) \Phi}_2 & \le \sum_{\ell \in Z_n}
\norm{{\check\chi}_j (H_G -
  \la)^{-1}{{\check\chi}_\ell}} \norm{\eta}_2 \cr
   & \le \frac C{t -r} \sum_{\ell \in Z_n} \ee^{-\alpha |j - \ell|} (1 +
  \sqrt{\la_k}) \cr
  & \le C \frac{\sqrt r}{t - r} \sum_{\ell \in Z_n} \ee^{-\alpha |j -
\ell|}.
   \label{(A.19)}
\end{align*}
From H\"older's inequality we obtain $ \norm{ {\check\chi}_j (1-{\xi_n})
{{\Phi}}}_1 \le \norm{
{\check\chi}_j (1-{\xi_n}) {{\Phi}}}_2$, and this completes the proof.
\end{proof}

\section{A general result for discrete eigenvalues}
\label{gen-result}

Here we show that it is a very general property that eigenfunctions of the
Dirichlet Laplacian corresponding to isolated eigenvalues of finite multiplicity
are integrable.
This is more general than what is proved in the previous sections and applies
also to eigenvalues above the infimum of the essential spectrum, lying in
gaps of the essential spectrum. However, this general result does not give
information about $L_1$-bounds of the eigenfunctions, which is the most
important point in the previous sections (and in fact in the present paper).

It would be of interest to obtain $L_1$-estimates for
eigenfunctions associated with a discrete eigenvalue located in a
gap of $\sigmaess(H_\Omega)$ above the infimum of the essential
spectrum.

Let $H$ ($=-\Delta$) denote the Laplacian in $L_2(\R^d)$, and let
$\Omega\sse\R^d$ be open. Then
the $C_0$-semigroup generated by $-H_\Omega$ is dominated by the
$C_0$-semigroup generated by $-H$; see e.g. \cite{PS},
\cite[Theorem 2.1.6]{dav}, \cite[Theorem B.2]{st-vo-96}.
This implies that the $C_0$-semigroup
generated by $-H_\Omega$ is associated with an integral kernel satisfying a
Gaussian estimate.

\begin{theorem}
Let $\Omega\sse\R^d$ be an open set, and assume that $ \la$ is an
eigenvalue of $H_\Omega$ of finite algebraic multiplicity (or in
other words, $ \la$ is an isolated point of the spectrum of
$H_\Omega$ and an eigenvalue of $H_\Omega$ of finite
multiplicity). Then the eigenspace corresponding to $ \la$ is a
subspace of $L_p(\Omega)$, for all $p\in[1,\infty]$.
\end{theorem}

\begin{proof}
For $2\le p\le\infty$, the assertion follows from the facts that
the eigenspace corresponding to $ \la$ is invariant under the
$C_0$-semigroup generated by $H_\Omega$ and that the Gaussian
estimate of the semigroup kernel implies that $L_2(\Omega)$ is
mapped to $L_p(\Omega)$ for positive times ($p$-$q$-smoothing
property of the semigroup for $1\le p\le q\le\infty$). For $1\le
p<2$ we recall from \cite[Corollary 4.3 and Example
5.1(a)]{are-94} that the component $\rho_\infty(-H_{\Omega,p})$ of
the $L_p$-resolvent set of $-H_\Omega$ containing the right
half-plane (which for $p=2$ is equal to the resolvent set of
$H_\Omega$, because the spectrum is a subset of $(-\infty,0]$) is
independent of $1\le p<\infty$. Moreover, it is shown in
\cite{are-94} that the resolvents are consistent in
$\rho(H_{\Omega}$).

Now, the hypothesis states that $\la$ is a pole of the resolvent
of $H_\Omega$, with finite rank residuum (which is just the
corresponding
spectral projection). Then we conclude from \cite[Theorem
1.3]{he-vo-87} (see also \cite{aut-83}) that $ \la$ is an
eigenvalue of finite algebraic multiplicity of $H_{\Omega,p}$
(where $-H_{\Omega,p}$ denotes the generator of the
$L_p$-semigroup), for all $1\le p<\infty$, and that the range of
the residuum is independent of $p$. As the eigenspace
corresponding to the eigenvalue $\la$ is just the range of the
residuum we conclude that it is a subspace of $L_p(\Omega)$ for
all $p\in[1,\infty)$.
\end{proof}

\section{Heat content and heat trace} \label{content-trace}

 We let $(\ee^{-tH_\Omega} ; t \ge 0)$ denote the $C_0$-semigroup generated
by $H_\Omega$ in $L_2(\Omega)$. For $f$ continuous and bounded,
$(\ee^{-tH_\Omega}f ; t \ge 0)$ provides a (weak) solution of the
initial boundary value problem for the heat equation, given by
\begin{equation*}
   \frac{\partial u}{\partial t} = \Delta u\qquad  (x\in \Omega, \ t>0),
\end{equation*}
where $\lim_{t \downarrow 0} u(\cdot;t) = f$, locally uniformly,
and $u(.;t) = 0$ on $\partial\Omega$ for $t > 0$ in the usual weak
sense that $u(\cdot;t) \in \Hnought^1(\Omega)$. As is well-known
\cite{Gr}, there is a smooth function
\begin{equation*}
   \Omega \times \Omega \times (0,\infty) \ni (x,y;t) \mapsto p_{\Omega}(x,y;t),
\end{equation*}
called the {\it Dirichlet heat kernel for $\Omega$}, such that
\begin{equation*}
    (\ee^{-tH_\Omega} f)(x) = \int_\Omega p_\Omega(x,y;t) f(y)\di y
 \qquad (x \in \Omega, \ t > 0).
\end{equation*}
In particular
\begin{equation*}
     u(x;t) := \int_{\Omega}p_{\Omega}(x,y;t)\di y \qquad (x\in\Omega, \ t > 0)
\end{equation*}
solves the above initial boundary value problem
for the constant function $f = 1$. At regular boundary
points $x_0 \in \partial \Omega$ we have $u(x;t) \to 0$ as $\Omega \ni x \to x_0$, for
any $t > 0$. Physically, $u(x;t)$ represents the temperature at a
point $x$ at time $t$ if $\Omega$ initially has constant temperature $1$,
while the boundary is kept at temperature $0$ for all $t>0$.

 The \emph{heat content} of $\Omega$ at time $t>0$ is defined by
\begin{equation*}
       Q_{\Omega}(t) := \int_{\Omega} u(x;t)\di x =
    \int_{\Omega}\int_{\Omega}
                p_\Omega(x,y;t) \di y \di x\qquad(t>0).
\end{equation*}
This quantity has been studied extensively in the general setting
of open bounded sets with smooth boundaries in complete Riemannian
manifolds. See for example \cite{vdBG1,G}.

For $H_\Omega$ with compact resolvent,
 we let $(\Phi_k)_{k\in\N}$ denote an orthonormal basis of real
 eigenfunctions of $H_\Omega$, associated with the
 increasing sequence of eigenvalues $(\la_k)_{k\in\N}$. Then
\begin{equation}\label{e5.6}
    p_{\Omega}(x,y;t) = \sum_{k=1}^{\infty} \ee^{-t\lambda_k}\Phi_k(x)\Phi_k(y),
\end{equation}
in the sense that
\[
\ee^{-tH_\Omega}f=\sum_{k=1}^\infty\ee^{-t\la_k}
\int_\Omega\Phi_k(y)f(y)\di y \, \Phi_k
\]
for all $f\in L_2(\Omega)$, with convergence of the sum in $L_2(\Omega)$.
Assuming in addition that $\sum_{k=1}^{\infty} \ee^{-t\lambda_k}
\norm{\Phi_k}_1^2<\infty$, one obtains that the series \eqref{e5.6} also
converges absolutely in $L_1(\Omega\times\Omega)$, and thus
\begin{equation}\label{e5.7}
    Q_{\Omega}(t) =
\sum_{k=1}^{\infty} \ee^{-t\lambda_k}\left(\int_{\Omega}\Phi_k(x)\di x\right)^2
\le \sum_{k=1}^{\infty} \ee^{-t\lambda_k} \norm{\Phi_k}_1^2.
\end{equation}

The \emph{trace} of the heat semigroup, denoted by $Z_{\Omega}(t)$ and defined
by
\begin{equation*}
Z_{\Omega}(t) := \sum_{k=1}^{\infty} \ee^{-t\lambda_k}
             = \int_{\Omega}p_{\Omega}(x,x;t) \di x,
\end{equation*}
 has been studied in great detail too (\cite{G}).
 It is well-known that heat content
or heat trace may be finite for all $t>0$ even if the volume of
$\Omega$ is infinite. See for example \cite{vdBD} for an early
paper on this subject. The main result of this section reads as
follows.
\begin{theorem}\label{5.1} Let $\Omega$ be an open set in $\R^d$
such that $H_\Omega$ has compact resolvent.
Then $Z_{\Omega}(t)<\infty$ for all $t>0$ if and only if
$Q_{\Omega}(t)<\infty$ for all $t>0$. In either case we have that
both
\begin{equation}\label{e5.9}
Z_{\Omega}(t)\le (2\pi t)^{-d/2}Q_{\Omega}(t/2),
\end{equation}
and
\begin{equation}\label{e5.10}
Q_{\Omega}(t) \le \mC
\Bigl(\lambda_1^{-3d/2}t^{-d}Z_{\Omega}(t/6)^3
+\lambda_1^{(6-9d)/2}t^{3-4d}Z_\Omega(t/2)\Bigr),
\end{equation}
where $\mC$ is a constant depending upon $d$ only.
\end{theorem}

In the proof of this result it will be shown that the hypothesis that
$Z_\Omega(t)<\infty$ for all $t>0$ implies that $\sum_{k=1}^{\infty}
\ee^{-t\lambda_k}
\norm{\Phi_k}_1^2<\infty$ for all $t>0$, and therefore the expression for
$Q_\Omega(t)$ stated in \eqref{e5.7} is valid.

We will need the following lemma where we use the above notation and
the assumptions of Theorem~\ref{5.1}.
\begin{lemma}\label{5.2} For any $T>0$, we have
\begin{equation}\label{e5.17}
N_{2\lambda_k}\le Z_{\Omega}(T)\ee^{2T\lambda_k}.
\end{equation}
\end{lemma}

\begin{proof} From
\begin{equation*}
Z_{\Omega}(T)\ge\sum_{j=1}^k \ee^{-T\lambda_j}\ge
\sum_{j=1}^k \ee^{-T\lambda_k} = k \ee^{-T\lambda_k},
\end{equation*}
we get
\begin{equation*}
\lambda_k \ge T^{-1}\log \frac{k}{Z_{\Omega}(T)},
\end{equation*}
and thus
\begin{equation*}
N_{2\lambda_k}= |\{j:\lambda_j\le2\lambda_k\}| \le \Bigl|\Bigl\{j
: T^{-1}\log\frac{j}{Z_{\Omega}(T)} \le2\lambda_k\Bigr\}\Bigr| \le
Z_{\Omega}(T) \ee^{2T\lambda_k}.
\end{equation*}
which concludes the proof of \eqref{e5.17}.
 \end{proof}

\begin{proof}[Proof of Theorem \ref{5.1}]
The proof of \eqref{e5.9} is an
immediate consequence of Lemma 2.6
in \cite{vdBD}.

The proof of \eqref{e5.10} relies
on the $L_1$-bounds for the eigenfunctions in Theorem~\ref{Theorem 0.1} or
Corollary~\ref{2.5} (with $\theta=1$) which gives the estimate
\begin{equation}\label{e5.11}
\norm{\Phi_k}_1^2\le
      C\lambda_1^{-3d/2}\lambda_k^{d}\left( (\log N_{2\lambda_k})^dN_{2\lambda_k}
         +\Bigl(\frac{\la_k}{\la_1}\Bigr)^{3(d-1)}\right).
\end{equation}
Hence by \eqref{e5.7} and \eqref{e5.11} we have that
\begin{equation}\label{e5.12}
 Q_{\Omega}(t)\le C\lambda_1^{-3d/2}\sum_{k=1}^{\infty}
 \ee^{-t\lambda_k}\lambda_k^{d}\left((\log
 N_{2\lambda_k})^d N_{2\lambda_k}+\Bigl(\frac{\la_k}{\la_1}\Bigr)^{3(d-1)}
\right).
\end{equation}
It is easily seen that $\log x \le d x^{1/d}$
 ($x\ge 1$) so that
\begin{equation}\label{e5.13}
Q_{\Omega}(t)\le C\lambda_1^{-3d/2}\sum_{k=1}^{\infty}
 \ee^{-t\lambda_k}\lambda_k^{d}\left(d^d
N^2_{2\lambda_k}+ \Bigl(\frac{\la_k}{\la_1}\Bigr)^{3(d-1)}\right).
\end{equation}
The following inequality is useful to bound the polynomial terms
in $\lambda_k$ in \eqref{e5.13}:
\begin{equation}\label{e5.14}
 \ee^{-tx}x^{\alpha}\le (\alpha/\ee)^{\alpha}t^{-\alpha}\qquad (x>0,\ t>0,\
\alpha>0).
\end{equation}
The application of this inequality with $x=\lambda_k/2$ and
$\alpha=4d-3$ gives that
\begin{equation*}\label{e5.15}
\sum_{k=1}^{\infty} \ee^{-t\lambda_k}\lambda_k^{4d-3}\le
((8d-6)/\ee)^{4d-3}t^{3-4d}Z_{\Omega}(t/2).
\end{equation*}
Hence the second term in \eqref{e5.13} is bounded by
\begin{equation}\label{e5.16}
((8d-6)/\ee)^{4d-3}C\lambda_1^{(6-9d)/2}t^{3-4d}Z_{\Omega}(t/2).
\end{equation}
By Lemma~\ref{5.2}, the first term in \eqref{e5.13} is bounded by
\begin{align}\label{e5.21}
\begin{split}
 d^dC\lambda_1^{-3d/2}& Z_\Omega(T)^2\sum_{k=1}^{\infty}
 \ee^{-t\lambda_k+4T\lambda_k}\lambda_k^{d}\\ 
& \le d^d
(6d/\ee)^dC\lambda_1^{-3d/2} Z_\Omega(T)^2t^{-d}\sum_{k=1}^{\infty}
 \ee^{-5t\lambda_k/6+4T\lambda_k},
\end{split}
\end{align}
where we have used \eqref{e5.14} with $x=\la_k/6$ and $\alpha=d$. We 
next choose
$T=t/6$ so that the right hand side in \eqref{e5.21} equals
\begin{equation}\label{e5.22}
(6d^2/\ee)^dC\lambda_1^{-3d/2} Z_{\Omega}(t/6)^3t^{-d}.
\end{equation}
Putting the two contributions under \eqref{e5.16} and
\eqref{e5.22} together one obtains the bound under \eqref{e5.10} with
\begin{equation*}
\mC= C\max\{(6d^2/\ee)^d,((8d-6)/\ee)^{4d-3}\}.
\qedhere
\end{equation*}
\end{proof}

We finally give a proof of Theorem \ref{2.12}.

\begin{proof}[Proof of Theorem \ref{2.12}]

Using the domain monotonicity of the Dirichlet heat kernel $0 \le
p_\Omega(x,y;t) \le p_{\R^d}(x,y;t)$ and the Schwarz inequality,
we first obtain
\begin{align*}
  \ee^{-t\la}\Phi(x) & = \ee^{-tH_\Omega}\Phi(x) = \int_\Omega p_\Omega(x,y;t)
  \Phi(y) \di y \\
  & \le  \int_\Omega p_\Omega(x,y;t) |\Phi(y)| \di y
   \le \left(\int_\Omega p^2_{\R^d}(x,y;t) \di y \right)^{1/2} \norm{\Phi}_2,
\end{align*}
 where $ \int_\Omega  p^2_{\R^d}(x,y;t) \di y =  (8 \pi t)^{-d/2} $
 since $ (2\pi t)^{-d/2} \int_{\R^d} \ee^{- |x-y|^2 /(2t)} \di y = 1$.
The choice of $t$ as $t := \frac{d}{4\lambda}$ then leads to the desired
estimate.
Furthermore,  $\norm{\Phi}_2^2 =  \int_\Omega |\Phi|^2 \di x
     \le \norm{\Phi}_\infty \norm{\Phi}_1$,
so that \eqref{(1)} implies \eqref{(2)}.
\end{proof}

{
}
\end{document}